\documentclass[final]{siamltex}

\usepackage{amsmath}
\usepackage{amsfonts}
\usepackage{amssymb}
\usepackage{latexsym}
\usepackage{enumerate}
\usepackage{array}
\newcommand{\beq}{\begin{equation}}
\newcommand{\eeq}{\end{equation}}
\newcommand{\bea}{\begin{eqnarray}}
\newcommand{\eea}{\end{eqnarray}}
\newcommand{\bean}{\begin{eqnarray*}}
\newcommand{\eean}{\end{eqnarray*}}

\newcommand {\mat}      [1] {\left[\begin{array}{#1}}
\newcommand {\rix}          {\end{array}\right]}

\newcommand\R{{\mathbb{R}}}

\newcommand\C{\mathbb{C}}
\newcommand{\Cn}{\C^{n}}

\newcommand{\Cnn}{\C^{n\times n}}

\newcommand{\si}{^{(i)}}

\newcommand{\sip}{^{(i+1)}}

\newcommand{\gref}[1]{{(\ref{#1})}}
\newcommand{\DS}{\displaystyle}

\newtheorem{assum}[theorem]{Assumption}
\newtheorem{alg}[theorem]{Algorithm}
\newtheorem{exmp}[theorem]{Example}
\newtheorem{rem}[theorem]{Remark}
\title{The calculation of the distance to a nearby defective matrix}

\author{Melina A. Freitag\thanks{Department of Mathematical Sciences,
        University of Bath, Claverton Down, BA2 7AY, United Kingdom (email:{\tt m.freitag@maths.bath.ac.uk}). Corresponding author. This author was supported by GWR, UK.}
\and Alastair Spence\thanks{Department of Mathematical Sciences,
        University of Bath, Claverton Down, BA2 7AY, United Kingdom (email:{\tt as@maths.bath.ac.uk})}.}

\begin{document}

\maketitle

\begin{abstract}
In this paper a new fast algorithm for the computation of the distance of a matrix to a nearby defective matrix is presented. The problem is formulated following Alam \& Bora (Linear Algebra Appl., 396 (2005), pp.~273--301) and reduces to finding when a parameter-dependent matrix is singular subject to a constraint. The solution is achieved by an extension of the Implicit Determinant Method introduced by Spence \& Poulton (J. Comput. Phys., 204 (2005), pp.~65--81). Numerical results for several examples illustrate the performance of the algorithm.
\end{abstract}



\pagestyle{myheadings}
\thispagestyle{plain}
\markboth{M. A. FREITAG AND A. SPENCE}{DISTANCE TO A NEARBY DEFECTIVE MATRIX}
\section{Introduction}

Let $A$ be a complex $n\times n$ matrix with $n$ distinct eigenvalues. It is a classic problem in numerical linear algebra to find
\[
d(A) = \inf\{\|A-B\|,\quad B \quad\text{is a defective matrix}\},
\]
where $\|\cdot\| = \|\cdot\|_F$ or $\|\cdot\| = \|\cdot\|_2$. Hence $d(A)$ is the distance of the matrix $A$ to the set of matrices which have a Jordan block of at least dimension $2$. In this paper we present a fast numerical method to find a nearby defective matrix. We formulate the problem as a real $3$-dimensional nonlinear system which is solved by Newton's method. Though not guaranteed to find the nearest defective matrix, since Newton's method provides no such guarantees, in all the examples considered our method did, in fact, find the nearest defective matrix and hence $d(A)$ was computed.

The distance of a matrix to a defective matrix is linked with the sensitivity analysis of eigenvalues. The condition number of a simple eigenvalue $\lambda$ is given by $1/|y^H x|$, (see \cite{Wilkinson65}) where $x$ and $y$ are normalised right and left eigenvectors corresponding to $\lambda$. For a defective eigenvalue we have $y^H x=0$ and hence the condition number of the eigenvalue is infinite. But it is well-known that even if the eigenvalues of a matrix are simple and well-separated from each other, they can be ill-conditioned. Hence the measure of the distance $d(A)$ of a matrix $A$ to a defective matrix $B$ is important for determining the sensitivity of the eigendecomposition. There is a very informative discussion and history of this problem in \cite{ABBO09}, where the contributions of Demmel \cite{Demmel83, Demmel86} and Wilkinson \cite{Wilkinson84, Wilkinson86} are discussed in detail. Another important paper is that by Lippert and Edelman \cite{LippertEdelman99} who use ideas from differential geometry and singularity theory to discuss the sensitivity of double eigenvalues. In particular, they present a condition that measures the ill-conditioning of a matrix with a $2$-dimensional Jordan block. The key paper that provides the solution to the nearest defective matrix is that of Alam and Bora \cite{AlamBora2005} who provide both theory and an algorithm based on pseudospectra.

Following Trefethen and Embree \cite{TrefEmbree05}, the $\varepsilon$-pseudospectrum $\Lambda_{\varepsilon}(A)$ of a matrix $A$ is given by
\[
\Lambda_{\varepsilon}(A) = \{ \sigma_{\min}(A - zI)<\varepsilon\},
\]
where $\varepsilon>0$ and $\sigma_{\min}$ denotes the smallest singular value. Equivalently
\[
\Lambda_{\varepsilon}(A) = \{z\in\C\,|\, \text{det}(A+E-zI)=0,\,\text{for some }\, E\in\Cnn\,\text{with}\, \|E\|<\varepsilon\}.
\]
If $\Lambda_{\varepsilon}(A)$ has $n$ components, then $A+E$ has $n$ distinct eigenvalues for all perturbation matrices $E\in\Cnn$ with $\|E\|<\varepsilon$ and hence $A+E$ is not defective. Alam and Bora \cite{AlamBora2005} take these ideas and seek the smallest perturbation matrix $E$ such that the pseudospectra of $A+E$ coalesce. They present the following theorem (see \cite[Theorem 4.1]{AlamBora2005} and \cite[Lemma 1]{ABBO09}).

\begin{theorem}
\label{th:def}
Let $A\in\Cnn$ and $z\in\C\setminus\Lambda(A)$, so that $A-z I$ has a simple smallest singular value $\varepsilon>0$ with corresponding left and right singular vectors $u$ and $v$ such that $(A-zI)v = \varepsilon u$. Then $z$ is an eigenvalue of $B = A-\varepsilon uv^H$ with geometric multiplicity $1$ and corresponding left and right eigenvectors $u$ and $v$ respectively. Furthermore, if $u^H v = 0$, then $z$ has algebraic multiplicity greater than one, hence it is a nonderogatory defective eigenvalue of $B$ and $\|A-B\|= \varepsilon$.
\end{theorem}

\begin{proof}
As $A-zI$ has unique smallest singular value $\varepsilon>0$ with corresponding left and right singular vectors $u$ and $v$, we have $\text{rank}(A-zI-\varepsilon uv^H)=n-1$ and hence $z$ is an eigenvalue of $B =A-\varepsilon uv^H$ with geometric multiplicity $1$. Further
\[
Bv = zv\quad\text{and}\quad u^H B = z u^H,
\]
and $u$ and $v$ are left and right eigenvectors such that $u^H v = 0$. So $z$ is a multiple eigenvalue of $B$ and hence $z$ is a nonderogatory defective eigenvalue of $B$.
\end{proof}

Theorem \ref{th:def} leads to the result $E:= -\varepsilon uv^H$ so that $B = A+E$ is a defective matrix and $d(A)=\varepsilon$, since $v^Hv=u^Hu=1$. One drawback of the algorithm in \cite{AlamBora2005} is that it is rather expensive since it involves repeated calculation of pseudospectra. Also a decision on when two pseudospectral curves coalesce is required. In \cite{ABBO09} a method based on calculating lowest generalised saddle points of singular values is described. This has the advantage that it is able to deal with the nongeneric case when $A-\varepsilon uv^H$ is ill-conditioned. We shall present a straightforward, yet elegant and very fast method that deals with the generic case when  $A-\varepsilon uv^H$ is well-conditioned. 

Using the notation of Theorem \ref{th:def} the problem is to find $z\in\C$, $u,v\in\Cn$ and $\varepsilon\in\R$ such that 
\begin{eqnarray}
(A-zI)v - \varepsilon u=0\label{eq:sv1}\\
\varepsilon v - (A-zI)^Hu=0\label{eq:sv2}
\end{eqnarray}
and
\begin{equation}
\label{eq:sv3}
u^H v = 0.
\end{equation}
Following Theorem \ref{th:def} and Lippert and Edelman \cite[Sections 4 and 5]{LippertEdelman99} we make the following assumption.

\begin{assum}
\label{assum:0}
Assume $A-z I$ satisfies the conditions of Theorem \ref{th:def} and that $B = A-\varepsilon uv^H$ is well-conditioned. That is, with $z=\alpha+\beta i$, the $2\times 2$ matrix $\displaystyle\mat{cc}\varepsilon_{\alpha\alpha}&\varepsilon_{\alpha\beta}\\  \varepsilon_{\alpha\beta}&\varepsilon_{\beta\beta}\rix$ is well-conditioned, where $\varepsilon_{\alpha\alpha}$ denotes the second partial derivative of $\varepsilon$ with respect to $\alpha$, etc. (see \cite[Theorem 5.1 and Corollary 5.2]{LippertEdelman99}).
\end{assum}

The paper is organised as follows. Section \ref{sec:bg} contains some background theory and the derivation of the implicit determinant method for our problem. Section \ref{sec:newton} describes the Newton method applied to this problem and in Section \ref{sec:ex} we give numerical examples that illustrate the power of our approach.

\section{The implicit determinant method to find a nearby defective matrix}
\label{sec:bg}

In this section we describe some background theory and present our numerical approach to finding a nearby defective matrix, which is formulated as solving a $3$-dimensional real nonlinear system. We emphasise that, since our numerical method uses standard Newton's method to solve the nonlinear system, we cannot guarantee to find the nearest defective matrix. However, a more sophisticated nonlinear solver may be used if greater reliability were sought. We do not do this here because in all our examples the nearest defective matrix was found using standard Newton's method. 

First, we formulate the problem following Alam and Bora \cite[Section 4]{AlamBora2005}. Equations \gref{eq:sv1}-\gref{eq:sv2} can be written as 
\beq
\label{eq:full}
\mat{cc}-\varepsilon I&A-zI\\ (A-zI)^H&-\varepsilon I\rix\mat{c}u\\ v\rix = \mat{c}0\\ 0\rix.
\eeq
Set $z = \alpha+i\beta$, $x = \displaystyle\mat{c}u\\ v\rix$ and introduce the Hermitian matrix
\beq
\label{eq:Habe}
K(\alpha,\beta,\varepsilon) = \mat{cc}-\varepsilon I&A-(\alpha +i\beta)I\\ (A-(\alpha +i\beta)I)^H&-\varepsilon I\rix.
\eeq
Clearly, $x$ is both a right and left null vector of $K(\alpha,\beta,\varepsilon)$. 
The following Lemma follows immediately from Assumption \ref{assum:0}.

\begin{lemma}
\label{lem:1}
Let $\varepsilon>0$ satisfy the conditions in Theorem \ref{th:def}. Furthermore, let $z = \alpha+i\beta$ be such that $K(\alpha,\beta,\varepsilon)x=0$, where $x = \displaystyle\mat{c}u\\ v\rix$. Then $\text{dim}\,\text{ker}\, K(\alpha,\beta,\varepsilon) = 1$.
\end{lemma}

\noindent We now introduce an algorithm to find the critical values of $\alpha$, $\beta$ and $\varepsilon$ such that the Hermitian matrix $K(\alpha,\beta,\varepsilon)$ is singular and the constraint on $x$ given by \gref{eq:sv3} is satisfied. We use the implicit determinant method, introduced in \cite{SpencePoulton05} to find photonic band structure in periodic materials such as photonic crystals. In \cite{FrSp09} the implicit determinant method was used to find a $2$-dimensional Jordan block in a Hamiltonian matrix in order to calculate the distance to instability. Here, in contrast to \cite{SpencePoulton05}, we have a three-parameter problem with a constraint to satisfy, and apply the method to a classic problem in numerical linear algebra. 

First we introduce a bordered matrix $M$. The next theorem gives conditions to ensure that this matrix is nonsingular.

\begin{theorem}
\label{th:abcd}
Let $(\alpha^*,\beta^*,\varepsilon^*,x^*)$ solve 
\[
K(\alpha,\beta,\varepsilon)x = 0,\quad x\neq 0,
\]
so that $\text{dim}\,\text{ker}\,K(\alpha^*,\beta^*,\varepsilon^*) = 1$ and $x^*\in \text{ker}(K(\alpha^*,\beta^*,\varepsilon^*))\setminus\{0\}$. For some $c\in\C^{2n}$ assume 
\[
c^Hx^*\neq 0.
\]
Then the Hermitian matrix 
\beq
\label{eq:M}
M(\alpha,\beta,\varepsilon) = \mat{cc} K(\alpha,\beta,\varepsilon)  &c\\c^H&0\rix
\eeq
is nonsingular at $\alpha = \alpha^*$, $\beta = \beta^*$, $\varepsilon=\varepsilon^*$.
\end{theorem}
\begin{proof}
This result follows from \cite[Lemma 2.8]{Keller77}.
\end{proof}

\noindent As $M(\alpha^*,\beta^*,\varepsilon^*)$ is nonsingular we have that $M(\alpha,\beta,\varepsilon)$ is nonsingular for $\alpha$, $\beta$ and $\varepsilon$ in the vicinity of 
$\alpha^*$, $\beta^*$ and $\varepsilon^*$. Now consider the following linear system
\beq
\label{eq:linsys}
\mat{cc}K(\alpha,\beta,\varepsilon)&c\\ c^H&0\rix\mat{c}x(\alpha,\beta,\varepsilon)\\ f(\alpha,\beta,\varepsilon)\rix=\mat{c}0\\ 1\rix,
\eeq
where $K(\alpha,\beta,\varepsilon)$ is given by \gref{eq:Habe}. Clearly, Theorem \ref{th:abcd} implies that both $x$ and $f$ are smooth functions of $\alpha$, $\beta$ and $\varepsilon$ near $(\alpha^*,\beta^*,\varepsilon^*)$.

Applying Cramer's rule to \gref{eq:linsys} we obtain
\[
f(\alpha,\beta,\varepsilon) = \frac{\text{det}\,K(\alpha,\beta,\varepsilon)}{\text{det}\,M(\alpha,\beta,\varepsilon)},
\]
and as $M(\alpha,\beta,\varepsilon)$ is nonsingular in the neighbourhood of $(\alpha^*,\beta^*,\varepsilon^*)$ by Theorem \ref{th:abcd} there is an equivalence between the zero eigenvalues of $K(\alpha,\beta,\varepsilon)$ (which we are looking for) and the zeros of $f(\alpha,\beta,\varepsilon)$. Hence, to find the values of $\alpha$, $\beta$ and $\varepsilon$ such that $\text{det}\,K(\alpha,\beta,\varepsilon)=0$ we seek the solutions of 
\beq
\label{eq:first}
f(\alpha,\beta,\varepsilon) = 0.
\eeq
If $f(\alpha^*,\beta^*,\varepsilon^*)=0$, the first row of system \gref{eq:linsys} gives
\beq
\label{eq:H}
K(\alpha^*,\beta^*,\varepsilon^*)x(\alpha^*,\beta^*,\varepsilon^*) = 0,
\eeq
that is, $x(\alpha^*,\beta^*,\varepsilon^*)= x^*$ is an eigenvector of $K(\alpha^*,\beta^*,\varepsilon^*)$ belonging to the eigenvalue zero. For the following derivation we use the notation
\beq
\label{eq:xy}
x(\alpha,\beta,\varepsilon) = \mat{c}u(\alpha,\beta,\varepsilon)\\ v(\alpha,\beta,\varepsilon)\rix.
\eeq
Note also that since $K(\alpha,\beta,\varepsilon)$ and $M(\alpha,\beta,\varepsilon)$ are Hermitian, $f(\alpha,\beta,\varepsilon)$ is real. Differentiating the linear system \gref{eq:linsys} with respect to $\alpha$ leads to
\beq
\label{eq:linsysalpha}
\mat{cc}K(\alpha,\beta,\varepsilon)&c\\ c^H&0\rix\mat{c}x_{\alpha}(\alpha,\beta,\varepsilon)\\ f_{\alpha}(\alpha,\beta,\varepsilon)\rix=\mat{c}v(\alpha,\beta,\varepsilon)\\u(\alpha,\beta,\varepsilon)\\ 0\rix,
\eeq
and the first row gives
\beq
\label{eq:Halpha}
K(\alpha,\beta,\varepsilon)x_{\alpha}(\alpha,\beta,\varepsilon) + cf_{\alpha}(\alpha,\beta,\varepsilon) = \mat{c}v(\alpha,\beta,\varepsilon)\\u(\alpha,\beta,\varepsilon)\rix.
\eeq
Multiplying this equation evaluated at $(\alpha^*,\beta^*,\varepsilon^*)$ from the left by the eigenvector ${x^*}^H$ of $K(\alpha^*,\beta^*,\varepsilon^*)$ gives
\[
f_{\alpha}(\alpha^*,\beta^*,\varepsilon^*) =\mat{cc} {u^*}^H & {v^*}^H\rix\mat{c}v^*\\u^*\rix = {u^*}^H v^*+{v^*}^H u^* = 2\text{Re}({u^*}^H v^*),
\]
where we have used ${x^*}^H c = 1$ from \gref{eq:linsys}. Similarly differentiating the linear system \gref{eq:linsys} with respect to $\beta$ gives
\beq
\label{eq:linsysbeta}
\mat{cc}K(\alpha,\beta,\varepsilon)&c\\ c^H&0\rix\mat{c}x_{\beta}(\alpha,\beta,\varepsilon)\\ f_{\beta}(\alpha,\beta,\varepsilon)\rix=i\mat{c}v(\alpha,\beta,\varepsilon)\\-u(\alpha,\beta,\varepsilon)\\ 0\rix.
\eeq
Again, evaluating at $(\alpha^*,\beta^*,\varepsilon^*)$ and multiplying by  ${x^*}^H$ from the left leads to
\[
f_{\beta}(\alpha^*,\beta^*,\varepsilon^*)=i\mat{cc}  {u^*}^H & {v^*}^H\rix\mat{c}v^*\\-u^*\rix=i({u^*}^H v^*-{v^*}^H u^*)= -2\text{Im}({u^*}^H v^*).
\]
Clearly
\[
f_{\alpha}(\alpha^*,\beta^*,\varepsilon^*)=0\quad\text{and}\quad f_{\beta}(\alpha^*,\beta^*,\varepsilon^*)=0\quad \Longleftrightarrow\quad {u^*}^H v^* = 0.
\]

Hence, we have reduced the problem of finding a solution to $\text{det}\,K(\alpha^*,\beta^*,\varepsilon^*)=0$ with ${u^*}^H v^* = 0$, to that of solving $g(\alpha,\beta,\varepsilon)=0$, where
\beq
\label{eq:fung}
g(\alpha,\beta,\varepsilon) = \mat{c}f(\alpha,\beta,\varepsilon)\\ f_{\alpha}(\alpha,\beta,\varepsilon)\\ f_{\beta}(\alpha,\beta,\varepsilon)\rix,
\eeq
which is three real nonlinear equations in three real unknowns. In the next section we describe the solution procedure using Newton's method.

\section{Newton's method applied to $g(\alpha,\beta,\varepsilon)=0$}
\label{sec:newton}

In this section we describe how to implement Newton's method for the nonlinear system $g(\alpha,\beta,\varepsilon)=0$. We also obtain a nondegeneracy condition that ensures nonsingularity of the Jacobian matrix of $g$ at the root, and hence confirms that Newton's method converges quadratically for a close enough starting guess. The nondegeneracy condition is shown to be equivalent to one introduced by Lippert and Edelman \cite{LippertEdelman99} for the conditioning of the $2$-dimensional Jordan block of  $B = A-\varepsilon uv^H$.

Newton's method applied to $g(\alpha,\beta,\varepsilon)$ is given by
\beq
\label{eq:newton}
G(\alpha\si,\beta\si,\varepsilon\si)\mat{c}\Delta\alpha\si\\\Delta\beta\si\\\Delta\varepsilon\si\rix = -g((\alpha\si,\beta\si,\varepsilon\si),
\eeq
where $\alpha\sip = \alpha\si+ \Delta\alpha\si$, $\beta\sip = \beta\si+ \Delta\beta\si$ and $\varepsilon\sip= \varepsilon\si+\Delta\varepsilon\si$, for $i=0,1,2\ldots$ until convergence, with a starting guess $(\alpha^{(0)},\beta^{(0)},\varepsilon^{(0)})$, where the Jacobian is 
\beq
\label{eq:jac}
G(\alpha\si,\beta\si,\varepsilon\si)=\mat{ccc}f_{\alpha}\si &f_{\beta}\si &f_{\varepsilon}\si \\f_{\alpha\alpha}\si &f_{\alpha\beta}\si &f_{\alpha\varepsilon}\si \\f_{\beta\alpha}\si &f_{\beta\beta}\si &f_{\beta\varepsilon}\si \rix,
\eeq
and all the matrix entries are evaluated at $(\alpha\si,\beta\si,\varepsilon\si)$. The values of $f\si$, $f_{\alpha}\si$ and $f_{\beta}\si$ are found using \gref{eq:linsys}, \gref{eq:linsysalpha} and \gref{eq:linsysbeta}. For the remaining values we differentiate \gref{eq:linsys}, \gref{eq:linsysalpha} and \gref{eq:linsysbeta} with respect to $\varepsilon$, that is,
\beq
\label{eq:linsysepsilon}
\mat{cc}K(\alpha,\beta,\varepsilon)&c\\ c^H&0\rix\mat{c}x_{\varepsilon}(\alpha,\beta,\varepsilon)\\ f_{\varepsilon}(\alpha,\beta,\varepsilon)\rix=\mat{c}x(\alpha,\beta,\varepsilon)\\ 0\rix,
\eeq
\beq
\label{eq:linsysalphaepsilon}
\mat{cc}K(\alpha,\beta,\varepsilon)&c\\ c^H&0\rix\mat{c}x_{\alpha\varepsilon}(\alpha,\beta,\varepsilon)\\ f_{\alpha\varepsilon}(\alpha,\beta,\varepsilon)\rix=\mat{c}v_{\varepsilon}(\alpha,\beta,\varepsilon)+ u_{\alpha}(\alpha,\beta,\varepsilon)\\u_{\varepsilon}(\alpha,\beta,\varepsilon)+ v_{\alpha}(\alpha,\beta,\varepsilon) \\0\rix,
\eeq
and
\beq
\label{eq:linsysbetaepsilon}
\mat{cc}K(\alpha,\beta,\varepsilon)&c\\ c^H&0\rix\mat{c}x_{\beta\varepsilon}(\alpha,\beta,\varepsilon)\\ f_{\beta\varepsilon}(\alpha,\beta,\varepsilon)\rix=\mat{c}iv_{\varepsilon}(\alpha,\beta,\varepsilon)+u_{\beta}(\alpha,\beta,\varepsilon)\\ -iu_{\varepsilon}(\alpha,\beta,\varepsilon)+v_{\beta}(\alpha,\beta,\varepsilon)\\0\rix
\eeq
in order to find $f_{\varepsilon}\si$, $f_{\alpha\varepsilon}\si$ and $f_{\beta\varepsilon}\si$. Furthermore, we differentiate \gref{eq:linsysalpha} with respect to $\alpha$ and $\beta$, that is 
\beq
\label{eq:linsysalphaalpha}
\mat{cc}K(\alpha,\beta,\varepsilon)&c\\ c^H&0\rix\mat{c}x_{\alpha\alpha}(\alpha,\beta,\varepsilon)\\ f_{\alpha\alpha}(\alpha,\beta,\varepsilon)\rix=2\mat{c}v_{\alpha}(\alpha,\beta,\varepsilon)\\u_{\alpha}(\alpha,\beta,\varepsilon)\\0\rix,
\eeq
and
\beq
\label{eq:linsysalphabeta}
\mat{cc}K(\alpha,\beta,\varepsilon)&c\\ c^H&0\rix\mat{c}x_{\alpha\beta}(\alpha,\beta,\varepsilon)\\ f_{\alpha\beta}(\alpha,\beta,\varepsilon)\rix=\mat{c}iv_{\alpha}(\alpha,\beta,\varepsilon)+v_{\beta}(\alpha,\beta,\varepsilon)\\-iu_{\alpha}(\alpha,\beta,\varepsilon)+u_{\beta}(\alpha,\beta,\varepsilon)\\0\rix,
\eeq
to compute $f_{\alpha\alpha}\si$ and $f_{\alpha\beta}\si = f_{\beta\alpha}\si$. Finally, differentiate \gref{eq:linsysbeta} with respect to $\beta$ to get 
\beq
\label{eq:linsysbetabeta}
\mat{cc}K(\alpha,\beta,\varepsilon)&c\\ c^H&0\rix\mat{c}x_{\beta\beta}(\alpha,\beta,\varepsilon)\\ f_{\beta\beta}(\alpha,\beta,\varepsilon)\rix=2i\mat{c}v_{\beta}(\alpha,\beta,\varepsilon)\\-u_{\beta}(\alpha,\beta,\varepsilon)\\0\rix,
\eeq
Therefore, in order to evaluate the components of $G(\alpha\si,\beta\si,\varepsilon\si)$ and $g(\alpha\si,\beta\si,\varepsilon\si)$ we only need to solve the linear systems above, which, importantly, all have the same Hermitian system matrix $M(\alpha\si,\beta\si,\varepsilon\si)$. Hence only one LU factorisation of $M(\alpha\si,\beta\si,\varepsilon\si)$ per iteration in Newton's method is required. Note that Newton's method itself is only carried out in three dimensions. Next, we state the Newton method algorithm for this problem.

\begin{alg}[Newton's method]
\label{alg:newton}
Given $(\alpha^{(0)},\beta^{(0)},\varepsilon^{(0)})$ and $c\in\Cn$ such that 
$M(\alpha^{(0)},\beta^{(0)},\varepsilon^{(0)})$ is nonsingular; set $i=0$:
\begin{enumerate}[(i)]
\item Solve \gref{eq:linsys} and \gref{eq:linsysalpha} and \gref{eq:linsysbeta} in order to evaluate
\[
g(\alpha\si,\beta\si,\varepsilon\si) = \mat{c}f(\alpha\si,\beta\si,\varepsilon\si)\\f_{\alpha}(\alpha\si,\beta\si,\varepsilon\si)\\f_{\beta}(\alpha\si,\beta\si,\varepsilon\si) \rix.
\]
\item Solve \gref{eq:linsysepsilon}, \gref{eq:linsysalphaalpha}, \gref{eq:linsysalphabeta}, \gref{eq:linsysbetabeta}, \gref{eq:linsysalphaepsilon} and \gref{eq:linsysbetaepsilon} in order to evaluate the Jacobian $G(\alpha\si,\beta\si,\varepsilon\si)$ given by \gref{eq:jac}.
\item Newton update: Solve \gref{eq:newton} in order to get $(\alpha\sip,\beta\sip,\varepsilon\sip)$
\item Repeat until convergence.
\end{enumerate}
\end{alg}

Finally we show, that provided a certain nondegeneracy condition holds, the Jacobian $G$ is nonsingular at the root. In the limit we have 
\beq
\label{eq:jaclimit}
G(\alpha^*,\beta^*,\varepsilon^*)=\mat{ccc}0 &0 &f_{\varepsilon}^* \\f_{\alpha\alpha}^* &f_{\alpha\beta}^* &f_{\alpha\varepsilon}^* \\f_{\beta\alpha}^* &f_{\beta\beta}^* &f_{\beta\varepsilon}^* \rix,
\eeq
since $f_{\alpha}^*= f_{\beta}^*=0$.
 
Multiplying the first row of \gref{eq:linsysepsilon} evaluated at $(\alpha^*,\beta^*,\varepsilon^*)$ from the left by ${x^*}^H$ gives
\[
f_{\varepsilon}(\alpha^*,\beta^*,\varepsilon^*) = {x^*}^H x^*> 0,\quad\text{since}\quad x^*\neq 0,
\]
(recall ${x^*}^H c = 1$ from \gref{eq:linsys}). Hence the Jacobian \gref{eq:jaclimit} is nonsingular if and only if 
\beq
\label{eq:Fstar}
F_{\alpha\beta}^*:=f_{\alpha\alpha}^*f_{\beta\beta}^*-{f_{\alpha\beta}^*}^2\neq 0.
\eeq 
With similar calculations as before we obtain
\beq
\label{eq:faafbb}
f_{\alpha\alpha}(\alpha^*,\beta^*,\varepsilon^*)  = 2{x^*}^H\mat{c}v_{\alpha}^*\\u_{\alpha}^*\rix,\quad f_{\beta\beta}(\alpha^*,\beta^*,\varepsilon^*)  = 2i{x^*}^H\mat{c}v_{\beta}^*\\-u_{\beta}^*\rix
\eeq
and
\beq
\label{eq:fab}
f_{\alpha\beta}(\alpha^*,\beta^*,\varepsilon^*)  = {x^*}^H\left(i\mat{c}v_{\alpha}^*\\-u_{\alpha}^*\rix+\mat{c}v_{\beta}^*\\u_{\beta}^*\rix\right).
\eeq
\begin{lemma}
\label{lem:F}
Under Assumption \ref{assum:0}, $F_{\alpha\beta}^*=f_{\alpha\alpha}^*f_{\beta\beta}^*-{f_{\alpha\beta}^*}^2\neq 0$.
\end{lemma}
\begin{proof}
If $\varepsilon$ is a simple singular value of $(A-(\alpha+\beta i)I)$, $\alpha,\beta\in\R$, so that 
\[
(A-(\alpha+\beta i)I)v = \varepsilon u,\quad (A-(\alpha+\beta i)I)^H u = \varepsilon v,
\] 
then (see Sun \cite{Sun88}) $\varepsilon$, $u$ and $v$ are smooth functions of $\alpha$ and $\beta$. Furthermore, Lippert and Edelman \cite[Theorem 3.1]{LippertEdelman99} show that if ${u^*}^H v^*=0$ then $\varepsilon_{\alpha}^*: = \varepsilon_{\alpha}(\alpha^*,\beta^*)=0$, $\varepsilon_{\beta}^*: = \varepsilon_{\beta}(\alpha^*,\beta^*)=0$ and $B=A-\varepsilon u^*{v^*}^H$ has a $2$-dimensional Jordan block. In addition, the ill-conditioning of the matrix $B$ is determined by the ill-conditioning of $E = \displaystyle\mat{cc} \varepsilon_{\alpha\alpha}^*& \varepsilon_{\alpha\beta}^*\\ \varepsilon_{\alpha\beta}^*& \varepsilon_{\beta\beta}^*\rix$, see \cite[Corollary 5.2]{LippertEdelman99}. Under Assumption \ref{assum:0} we have $\text{det}(E)\neq 0$. Recall \gref{eq:full} and \gref{eq:Habe}
where $\varepsilon=\varepsilon(\alpha,\beta)$, $v = v(\alpha,\beta)$, $u = u(\alpha,\beta)$ and $x=\DS\mat{c}u\\v\rix$. Taking the second derivatives with respect to $\alpha$ and $\beta$ and evaluating them at the root so that $\varepsilon_{\alpha}^*(\alpha,\beta)=\varepsilon_{\beta}^*(\alpha,\beta)=0$ we obtain
\[
K(\alpha^*,\beta^*,\varepsilon^*) x_{\alpha\alpha}^* - 2\mat{c}v_{\alpha}^*\\ u_{\alpha}^*\rix = \varepsilon_{\alpha\alpha}^*x^*,
\]
\[
K(\alpha^*,\beta^*,\varepsilon^*) x_{\beta\beta}^* + 2i\mat{c}-v_{\beta}^*\\ u_{\beta}^*\rix = \varepsilon_{\beta\beta}^*x^*,
\]
and
\[
K(\alpha^*,\beta^*,\varepsilon^*) x_{\alpha\beta}^* +\mat{c}-iv_{\alpha}^*-v_{\beta}^*\\ iu_{\alpha}^*-u_{\beta}^*\rix = \varepsilon_{\alpha\beta}^*x^*.
\]
Multiplying those three equations by the eigenvector ${x^*}^H$ of $K(\alpha^*,\beta^*,\varepsilon^*)$ from the left we obtain that
\[
f_{\alpha\alpha}^*=-({x^*}^H x^*)\varepsilon_{\alpha\alpha}^*,\quad f_{\beta\beta}^*=-({x^*}^H x^*)\varepsilon_{\beta\beta}^*\quad\text{and}\quad f_{\alpha\beta}^*=-({x^*}^H x^*)\varepsilon_{\alpha\beta}^*,
\]
where we have used \gref{eq:faafbb} and \gref{eq:fab}. Hence $F_{\alpha\beta}^*=f_{\alpha\alpha}^*f_{\beta\beta}^*-{f_{\alpha\beta}^*}^2\neq 0$ since $\text{det}(E)\neq 0$.
\end{proof}

\noindent In summary, Lemma \ref{lem:F} shows that when the defective matrix $B = A-\varepsilon uv^H$ is well-conditioned Algorithm \ref{alg:newton} should exhibit quadratic convergence for a close enough starting guess.

\begin{rem}
We note that $z = \alpha+\beta i$ is a saddle point of $f(\alpha,\beta)$ and hence $F_{\alpha\beta}^*=f_{\alpha\alpha}^*f_{\beta\beta}^*-{f_{\alpha\beta}^*}^2<0$. This property can in fact be checked and is observed in all the computational examples in Section \ref{sec:ex} (see last column of Tables \ref{tab:kahan6}-\ref{tab:grcar20}).  
\end{rem}

We would like to note that our algorithm depends on the starting guess and hence does not guarantee convergence to the nearest defective matrix but only to a nearby one. However, all the algorithms currently known in the literature only find nearby defective matrices (see, in particular the methods suggested in \cite{ABBO09}).

We would also like to point out some computational advantages of our method. Both the method in \cite{ABBO09} and our method provide a Newton method for finding a saddle point of $\sigma$ (in \cite{ABBO09} and \cite{LippertEdelman99}) or $f$. For our problem the derivatives of $f$ are particularly easy and simple to calculate. For any derivative a system with the same bordered Hermitian matrix \gref{eq:M} has to be solved - and we can get any 1st, 2nd or higher order derivatives by solving with the same matrix. Hence one matrix factorisation with costs of usually $\frac{2}{3}(2n+1)^3 \approx \frac{16}{3}n^3$ or less for sparse systems is sufficient. Other explicit methods for calculating first and second derivatives have been derived (see \cite{ABBO09} and \cite{LippertEdelman99}), which usually require a full SVD to be carried out, costing $21 n^3$ operations (see \cite{GolubvanLoan96}). Hence for large problems the implicit determinant method is more efficient. We show an example in the next section.

\subsection{Numerical examples}
\label{sec:ex}

We now illustrate the numerical performance of our method with two examples which are taken from \cite{AlamBora2005}. As has been mentioned earlier, since our method is based on Newton's method it finds a nearby defective matrix. We cannot guarantee it finds the closest defective matrix (this will depend on the starting guesses we use). However, in all cases considered here our method found the nearest defective matrix according to Alam and Bora \cite{AlamBora2005}.

\begin{exmp} 
\label{ex:Kahan}
Let $A\in\Cnn$ be the Kahan matrix \cite{TrefEmbree05}, which is given by
\[
A = \mat{ccccc}1&-c&-c&-c&-c\\ &s&-sc&-sc&-sc\\ &&s^2&-s^2 c& -s^2 c\\ &&&\ddots&\vdots\\ &&&&s^{n-1}\rix,
\]
where $s^{n-1}=0.1$ and $s^2+c^2=1$. We consider this matrix for $n=6,15,20$. As initial guesses we choose $\beta^{(0)} = 0$ and $\alpha^{(0)}= 0$ for $n=6$, $\alpha^{(0)}= 0.12$ for $n=15$ and $\alpha^{(0)}= 0.115$ for $n=20$. Further $\varepsilon^{(0)}=\sigma_{\min}$, $u^{(0)} = u_{\min}$ and $v^{(0)} = v_{\min}$, where $\sigma_{\min}$ is the minimum singular value of $A$ with corresponding left and right singular vectors $u_{\min}$ and $v_{\min}$. $x^{(0)}$ is determined from \gref{eq:xy} and  $c=x^{(0)}$. We stop the iteration once 
\[
\|g(\alpha\si,\beta\si,\varepsilon\si)\|<\tau,\quad\text{where}\quad\tau = 10^{-14}.
\]
\end{exmp}

\begin{table}[h!]
\begin{center}
\caption{Results for Example \ref{ex:Kahan}, $n = 6$.}
\vspace{0.1cm}
\footnotesize{
\begin{tabular}{c|c|c|c|c|c}
\hline
\hline
$i$&$\alpha\si$&$\beta\si$&$\varepsilon\si$&$\|g(\alpha\si,\beta\si,\varepsilon\si)\|$&$F_{\alpha\beta}\si$\\
\hline
0 &          0     &       0  & 9.9694e-03  &          -   &         -\\
1 &  1.3643e-01    &        0 &  1.2145e-02 &  8.1049e-02  & 3.9318e-01\\
2 &  1.3319e-01    &        0 &  7.1339e-04 &  3.9165e-02  &-1.0032e+00\\
3 &  1.2767e-01    &        0 &  4.9351e-04 &  4.3976e-03  &-4.5529e-01\\
4 &  1.2763e-01    &        0 &  4.7049e-04 &  8.2870e-05  &-4.3191e-01\\
5 &  1.2763e-01    &        0 &  4.7049e-04 &  4.7344e-08  &-4.3136e-01\\
6 &  1.2763e-01    &        0 &  4.7049e-04 &  5.3655e-15  &-4.3136e-01\\
\hline
\end{tabular}
}
\label{tab:kahan6}
\end{center}
\end{table}

Table \ref{tab:kahan6} shows the results for $n=6$. In this case the eigenvalues $1.5849\times 10^{-1}$ and $10^{-1}$ coalesce at $1.2763\times 10^{-1}$ for a value of $\varepsilon = 4.7049\times 10^{-4}$. The last column of Table \ref{tab:kahan6} shows the value of $F_{\alpha\beta}\si=f_{\alpha\alpha}\si f_{\beta\beta}\si-{f_{\alpha\beta}\si}^2$ (given by \gref{eq:Fstar}) and we see that the final value $F_{\alpha\beta}^*\neq 0$ at the root. The quadratic convergence rate is clearly observed.

\begin{table}[h!]
\begin{center}
\caption{Results for Example \ref{ex:Kahan}, $n = 15$.}
\vspace{0.1cm}
\footnotesize{
\begin{tabular}{c|c|c|c|c|c}
\hline
\hline
$i$&$\alpha\si$&$\beta\si$&$\varepsilon\si$&$\|g(\alpha\si,\beta\si,\varepsilon\si)\|$&$F_{\alpha\beta}\si$\\
\hline
0 &  1.2000e-01    &        0  & 4.7454e-04 &           - &           -\\
1 &  1.2042e-01    &        0  & 2.1767e-06 &  3.9203e-03 & -6.1848e-03\\
2 &  1.3116e-01    &        0  & 1.0065e-06 &  5.6943e-05 &  5.6071e-06\\
3 &  1.2833e-01    &        0  & 4.9786e-07 &  2.8915e-05 & -6.7015e-05\\
4 &  1.2865e-01    &        0  & 4.4839e-07 &  1.6066e-06 & -5.9016e-05\\
5 &  1.2865e-01    &        0  & 4.4850e-07 &  1.7737e-08 & -6.1975e-05\\
6 &  1.2865e-01    &        0  & 4.4850e-07 &  1.9014e-12 & -6.1957e-05\\
7 &  1.2865e-01    &        0  & 4.4850e-07 &  3.5480e-18 & -6.1957e-05\\
\hline
\end{tabular}
}
\label{tab:kahan15}
\end{center}
\end{table}
Table \ref{tab:kahan15} shows the results for $n=15$. In this case the eigenvalues $1.1788\times 10^{-1}$ and $1.3895\times 10^{-1}$ coalesce at $1.2865\times 10^{-1}$ for a value of $\varepsilon = 4.4850e-07\times 10^{-7}$.

\begin{table}[h!]
\begin{center}
\caption{Results for Example \ref{ex:Kahan}, $n = 20$.}
\vspace{0.1cm}
\footnotesize{
\begin{tabular}{c|c|c|c|c|c}
\hline
\hline
$i$&$\alpha\si$&$\beta\si$&$\varepsilon\si$&$\|g(\alpha\si,\beta\si,\varepsilon\si)\|$&$F_{\alpha\beta}\si$\\
\hline
0 &  1.1500e-01   &         0 &  1.3141e-04 &          -  &          -\\
1 &  1.1507e-01   &         0 &  1.1315e-07 &  1.2702e-03 & -7.9071e-04\\
2 &  1.2010e-01   &         0 &  3.2008e-08 &  3.4299e-06 & -5.8539e-09\\
3 &  1.1997e-01   &         0 &  1.8878e-08 &  2.8840e-07 & -4.3105e-07\\
4 &  1.2000e-01   &         0 &  1.9049e-08 &  2.2944e-08 & -4.6343e-07\\
5 &  1.2000e-01   &         0 &  1.9049e-08 &  7.3704e-13 & -4.6360e-07\\
6 &  1.2000e-01   &         0 &  1.9049e-08 &  2.1281e-17 & -4.6360e-07\\
\hline
\end{tabular}
}
\label{tab:kahan20}
\end{center}
\end{table}
Table \ref{tab:kahan20} shows the results for $n=20$. In this case the eigenvalues $1.1288\times 10^{-1}$ and $1.2743\times 10^{-1}$ coalesce at $1.2\times 10^{-1}$ for a value of $\varepsilon = 1.9049\times 10^{-8}$.

From the last columns in Tables \ref{tab:kahan6}-\ref{tab:kahan20} we see that the value of $F_{\alpha\beta}\si$ becomes smaller the larger the size of the Kahan matrix. This means the matrix $B(\varepsilon) = A-\varepsilon uv^H$ becomes increasingly ill-conditioned as $n$ increases. We also observe a corresponding deterioration in the rate of convergence of Newton's method as the value of $F_{\alpha\beta}\si$ becomes smaller, which is consistent with the theory. 

\begin{exmp} 
\label{ex:grcar}
Let $A\in\Cnn$ be the Grcar matrix taken from the Matlab gallery \verb|A = gallery('grcar',n)|, where $n=6,20$. The eigenvalues of $A$ appear in complex conjugate pairs and hence in this case two pairs of complex eigenvalues of $A$ coalesce at two boundary points of the pseudospectrum.

As initial guess for $n=6$ we take $\alpha^{(0)}= 0$, $\beta^{(0)} = -1$, $\varepsilon^{(0)}=0$, $u^{(0)} = u_{\min}$ and $v^{(0)} = v_{\min}$, where $u_{\min}$ and $v_{\min}$ are left and right singular vectors of $A-\beta^{(0)}iI$, corresponding to the smallest singular value. $x^{(0)}$ is determined from \gref{eq:xy}. The stopping condition is the same as in Example \ref{ex:Kahan}. For $n=20$ we take $\beta^{(0)} = -2.5$, the initial guesses for the remaining values are determined similarly. Furthermore $c = x^{(0)}$.
\end{exmp}

\begin{table}[h!]
\begin{center}
\caption{Results for Example \ref{ex:grcar}, $n = 6$.}
\vspace{0.1cm}
\footnotesize{
\begin{tabular}{c|c|c|c|c|c}
\hline
\hline
$i$&$\alpha\si$&$\beta\si$&$\varepsilon\si$&$\|g(\alpha\si,\beta\si,\varepsilon\si)\|$&$F_{\alpha\beta}\si$\\
\hline
0&            0& -1.0000e+00  &          0  &          -&-\\
1&   1.2141e+00&  -2.3756e+00 & 7.4297e-01 &  5.0533e-01&1.4186e-01\\
2&   1.1159e+00&  -1.4291e+00 &  9.5425e-02 &  2.2193e+01&-2.7279e+04\\
3&   1.0512e+00&  -1.9848e+00 &  4.3767e-01 &  5.2914e-01&-5.0768e+00\\
4&   8.0543e-01&  -1.5940e+00 &  1.4858e-01 &  4.1255e-01&-1.1717e+00\\
5&   7.5742e-01&  -1.5944e+00 &  2.1279e-01 &  8.6847e-02&-1.1323e+00\\
6&   7.5335e-01&  -1.5912e+00 &  2.1516e-01 &  5.5621e-03&-9.7810e-01\\
7&   7.5332e-01&  -1.5912e+00 &  2.1519e-01 &  4.2790e-05&-9.6333e-01\\
8&   7.5332e-01&  -1.5912e+00 &  2.1519e-01 &  2.4851e-09&-9.6323e-01\\
9&   7.5332e-01&  -1.5912e+00 &  2.1519e-01 &  1.5798e-16&-9.6323e-01\\
\hline
\end{tabular}
}
\label{tab:grcar6}
\end{center}
\end{table}
Table \ref{tab:grcar6} shows the results for $n=6$. The eigenvalue pairs $1.1391\pm 1.2303i$ and $3.5849\times 10^{-1}\pm 1.9501i$ coalesce at $7.5332\times 10^{-1}\pm 1.5912i$ for a value of $\varepsilon = 2.1519\times 10^{-1}$.

\begin{table}[h!]
\begin{center}
\caption{Results for Example \ref{ex:grcar}, $n = 20$.}
\vspace{0.1cm}
\footnotesize{
\begin{tabular}{c|c|c|c|c|c}
\hline
\hline
$i$&$\alpha\si$&$\beta\si$&$\varepsilon\si$&$\|g(\alpha\si,\beta\si,\varepsilon\si)\|$&$F_{\alpha\beta}\si$\\
\hline
0 &           0 & -2.5000e+00 &           0 &           0&0\\
1 &  9.5854e-02 & -2.3299e+00 & 1.7989e-02 &  1.3806e-01&9.9103e-01\\
2 &  1.3904e-01 & -2.2465e+00 & 1.3564e-03 &  3.2308e-02&-2.3623e-01\\
3 &  1.6141e-01 & -2.2042e+00 & 7.2914e-04 &  1.1930e-02&-1.5963e-01\\
4 &  1.5554e-01 & -2.1818e+00 & 4.5435e-04 &  3.4851e-03&-2.7982e-02\\
5 &  1.5338e-01 & -2.1815e+00 & 4.9060e-04 &  3.4265e-04&-2.4693e-02\\
6 &  1.5331e-01 & -2.1817e+00 & 4.9141e-04 &  2.3240e-05&-2.3956e-02\\
7 &  1.5331e-01 & -2.1817e+00 & 4.9141e-04 &  1.6942e-08&-2.4012e-02\\
8 &  1.5331e-01 & -2.1817e+00 & 4.9141e-04 &  4.6672e-14&-2.4012e-02\\
9 &  1.5331e-01 & -2.1817e+00 & 4.9141e-04 &  4.5263e-17&-2.4012e-02\\
\hline
\end{tabular}
}
\label{tab:grcar20}
\end{center}
\end{table}
Table \ref{tab:grcar20} shows the results for $n=20$. The eigenvalue pairs $1.0802\times 10^{-1}\pm 2.2253i$ and $2.1882\times 10^{-1}\pm 2.1132i$ coalesce at $1.5331\times 10^{-1}\pm 2.1817i$ for a value of $\varepsilon = 4.9141\times 10^{-4}$.

The last columns in Tables \ref{tab:grcar6}-\ref{tab:grcar20} show the values of $F_{\alpha\beta}\si$ which converge to values away from zero. The latter iterates illustrate almost quadratic convergence. Note that in this example $\beta\neq 0$, so $z = \alpha+\beta i$ is complex, though this makes no difference to the numerical method.

We finally give a comparison of the method in \cite{ABBO09} (see also \cite{LippertEdelman99}) to our method in terms of computational cost. Note that both the implicit determinant method - as all other methods known so far - only compute a \emph{nearby} defective matrix.
\begin{exmp} 
\label{ex:comparison}
Consider an $n\times n$ matrix with $n=1000$, which is an identity matrix apart from the upper left $6\times 6$ block which is the Kahan matrix from Example \ref{ex:Kahan}. As initial guess we take the estimate which was used in \cite{ABBO09}.
\end{exmp}

\begin{table}[h!]
\caption{Results for Example \ref{ex:comparison}, $n = 1000$, Implicit determinant method (left) and method in \cite{ABBO09} (right).}
\vspace{0.1cm}
\begin{minipage}[b]{0.5\linewidth}\centering
\footnotesize{
\begin{tabular}{c|c|c|c}
\hline
\hline
$i$&$\varepsilon\si$&$\alpha\si$&$\|g(\lambda\si,\varepsilon\si)\|$\\
\hline
0 &  4.6081e-04&   1.3175e-01  &            -\\
1 &  4.8049e-04&   1.2753e-01  &   2.3311e-03\\
2 &  4.7050e-04&   1.2763e-01  &   5.9278e-05\\
3 &  4.7049e-04&   1.2763e-01  &   5.7187e-08\\
4 &  4.7049e-04&   1.2763e-01  &   4.4987e-14\\
\hline
\end{tabular}
}
\end{minipage}
\hspace{0.2cm}
\begin{minipage}[b]{0.5\linewidth}\centering
\footnotesize{
\begin{tabular}{c|c|c|c}
\hline
\hline
$i$&$\varepsilon\si$&$\alpha\si$&$\|g(\lambda\si,\varepsilon\si)\|$\\
\hline
0&   4.6081e-04 &  1.3175e-01&   4.6623e-03\\
1&   4.7049e-04 &  1.2753e-01&   1.1568e-04\\
2&   4.7049e-04 &  1.2763e-01&   5.6904e-08\\
3&   4.7049e-04 &  1.2763e-01&   1.3769e-14\\
\hline
\end{tabular}
}
\end{minipage}
\label{tab:comp}
\end{table}
Table \ref{tab:comp} shows the results for this comparison. We see that both the method from \cite{ABBO09} and our new method exhibit very fast quadratic convergence to the desired nearby (in this case nearest) defective matrix (cf Table \ref{tab:kahan6}). However, the CPU times are very different. Whereas the method in \cite{ABBO09} (right Table) requires a CPU time of $24.3s$ the Implicit Determinant method only needs $5.4s$.

In summary, we note that both the method in \cite{ABBO09} and the method described in this paper do not guarantee convergence to the nearest defective matrix. For large problems the Implicit Determinant method seems to be faster, as it is not necessary to compute the full SVD at each step.  

We note that we also compared the method described in \cite{ABBO09} with the method described in this paper for Example \ref{ex:grcar}. For this problem it is particularly hard to find good starting values in order for both methods to converge. If we generate the starting guesses as described in \cite{ABBO09} we found that both methods stagnate or diverge - as for a small singular value $\varepsilon$ the Hessian becomes increasingly ill-conditioned. If we start with the starting guess described in Example \ref{ex:grcar} we found the Implicit Determinant method to converge (see Tables \ref{tab:grcar6} and \ref{tab:grcar20}) but the method described in \cite{ABBO09} not to be defined as the second derivatives of the singular values are undefined for $\varepsilon^{(0)}=0$. However, to give a fair comparison, the method described in \cite{ABBO09} describes a variant of Newton's method for the computation of a nearby defective matrix that is applicable to both generic and non-generic cases, whereas the method described in this paper only deals with the generic case, that is the computed singular value is assumed to be simple.

\section{Final remarks}

We have developed a new algorithm for computing a nearby defective matrix. Numerical examples show that this new technique performs well and gives quadratic convergence in the generic cases. 

Also, since the method is only Newton's method on a real $3$-dimensional nonlinear system (with only one LU factorisation required at each step) it is simple to apply and is significantly faster than the technique in \cite{AlamBora2005}.

However, as has already been mentioned, since it is based on Newton's method, convergence to the nearest defective matrix cannot be guaranteed, though in fact, in all the examples considered, convergence to the nearest defective matrix was achieved. Of course, a more sophisticated nonlinear solver, e.g. global Newton's method or a global minimiser, could be applied to \gref{eq:fung} if required.

Though our algorithm is designed to compute a nearby defective matrix in the
generic case (that is, there is a well-conditioned 2-dimensional Jordan
block) it has two features that enable it to recognise when the conditions
of Assumption 1.2 fail. First, if there is another singular value near
$\epsilon$ then the condition number of $M$ will be large. Second, if the
condition number of $M$ is small, but $F_{\alpha\beta}$ is close to zero at
the root, then  this indicates the presence of a nearby matrix with a Jordan
block of dimension greater than $2$. As such the algorithm in this paper
could be used to provide starting values for an alternative algorithm that could detect
a higher order singularity.

\bibliographystyle{siam}

\begin{thebibliography}{10}

\bibitem{AlamBora2005}
{\sc R.~Alam and S.~Bora}, {\em On sensitivity of eigenvalues and
  eigendecompositions of matrices}, Linear Algebra Appl., 396 (2005),
  pp.~273--301.

\bibitem{ABBO09}
{\sc R.~Alam, S.~Bora, R.~Byers, and M.~L. Overton}, {\em Characterization and
  construction of the nearest defective matrix via coalescence of
  pseudospectral components}, Linear Algebra Appl., 435 (2011), pp.~494 -- 513.

\bibitem{Demmel83}
{\sc J.~Demmel}, {\em {A numerical analyst's Jordan canonical form}}, PhD
  thesis, University of California at Berkeley, 1983.

\bibitem{Demmel86}
{\sc J.~W. Demmel}, {\em Computing stable eigendecompositions of matrices},
  Linear Algebra Appl., 79 (1986), pp.~163--193.

\bibitem{FrSp09}
{\sc M.~A. Freitag and A.~Spence}, {\em A {N}ewton-based method for the
  calculation of the distance to instability}, Linear Algebra Appl., 435
  (2011), pp.~3189 -- 3205.

\bibitem{GolubvanLoan96}
{\sc G.~H. Golub and C.~F. {Van Loan}}, {\em Matrix {C}omputations}, John
  Hopkins University Press, Baltimore, 3rd~ed., 1996.

\bibitem{Keller77}
{\sc H.~B. Keller}, {\em Numerical solution of bifurcation and nonlinear
  eigenvalue problems}, in Applications of {B}ifurcation {T}heory, P.~H.
  Rabinowitz, ed., Academic Press, New York, 1977, pp.~359--384.

\bibitem{LippertEdelman99}
{\sc R.~A. Lippert and A.~Edelman}, {\em The computation and sensitivity of
  double eigenvalues}, in Advances in computational mathematics ({G}uangzhou,
  1997), vol.~202 of Lecture Notes in Pure and Appl. Math., Dekker, New York,
  1999, pp.~353--393.

\bibitem{SpencePoulton05}
{\sc A.~Spence and C.~Poulton}, {\em Photonic band structure calculations using
  nonlinear eigenvalue techniques}, J. Comput. Phys., 204 (2005), pp.~65--81.

\bibitem{Sun88}
{\sc J.~G. Sun}, {\em A note on simple nonzero singular values}, J. Comput.
  Math., 6 (1988), pp.~258--266.

\bibitem{TrefEmbree05}
{\sc L.~N. Trefethen and M.~Embree}, {\em Spectra and pseudospectra}, Princeton
  University Press, Princeton, NJ, 2005.
\newblock The behavior of nonnormal matrices and operators.

\bibitem{Wilkinson65}
{\sc J.~H. Wilkinson}, {\em The {A}lgebraic {E}igenvalue {P}roblem}, Oxford
  University Press, Oxford, UK, 1965.

\bibitem{Wilkinson84}
{\sc J.~H. Wilkinson}, {\em Sensitivity of eigenvalues}, Utilitas Math., 25
  (1984), pp.~5--76.

\bibitem{Wilkinson86}
\leavevmode\vrule height 2pt depth -1.6pt width 23pt, {\em Sensitivity of
  eigenvalues. {II}}, Utilitas Math., 30 (1986), pp.~243--286.

\end{thebibliography}

\end{document}